\documentclass[12pt]{amsart}

\usepackage{amssymb}
\usepackage{amscd}

\title[almost symmetric numerical semigroups with maximal reduced type]{Classification of almost symmetric numerical semigroups with maximal reduced type} 
\author{Akihiro Sugawara}

\address{Graduate School of Science and Engineering, Yamagata University, Kojirakawa-machi 1-4-12, Yamagata 990-8560, Japan} 
\email{s251436d@st.yamagata-u.ac.jp}

\newtheorem{theorem}{Theorem}
\newtheorem{fact}{Fact} 
\newtheorem{lemma}{Lemma} 
\newtheorem{proposition}{Proposition}
\newtheorem{corollary}{Corollary} 

\newtheorem*{question}{Question}

\theoremstyle{definition}
\newtheorem{definition}{Definition}
\newtheorem{example}{Example} 
\newtheorem{remark}{Remark}

\def\e{{\mathrm e}}
\def\g{{\mathrm g}}

\def\m{{\mathrm m}}

\def\t{{\mathrm t}}

\def\F{{\mathrm F}}

\def\PF{{\mathrm {PF}}}
\def\rPF{{\mathrm {rPF}}}

\def\Ap{{\mathrm{ Ap}}}
\def\max{{\mathrm{ max}}}

\def\max{{\mathrm{max}}}
\def\msg{{\mathrm{ msg }}}

\def\Maximals{\mathrm{Maximals}_{\leq_S}}

\def\N{\mathbb{N}}

\def\Z{\mathbb{Z}}

\def\Ap{\mathrm{Ap}}

\def\int{\mathrm{int}}

\begin{document}
\begin{abstract}
    This paper determines almost symmetric numerical semigroups with maximal reduced type completely. In addition, this paper classifies MED-semigroups with maximal reduced type.
\end{abstract}

\maketitle

\section{Introduction}
Let $k$ be an algebraically closed field of characteristic zero, and let $(R,\mathfrak{m},k)$ be an equicharacteristic reduced one-dimensional complete local $k$-algebra. In \cite{zbMATH07423434}, Huneke, Maitra, and Mukundan introduced the reduced type $s(R)$ of $R$. The reduced type $s(R)$ is related to the Cohen--Macaulay type $\mathrm{type}(R)$, and it is known that $1 \leq s(R) \leq \mathrm{type}(R)$ (see \cite{arXiv:2406.15923}). We say that $R$ has maximal reduced type if $s(R)=\mathrm{type}(R)$. Maitra and Mukundan asked the following (see \cite[Question 2.5]{arXiv:2306.17069}):
\begin{question}\label{question}
    Can we classify when $R$ has maximal reduced type?
\end{question}

Let $S$ be a numerical semigroup. The numerical semigroup ring associated to $S$ is the ring of the form $k[[S]]=k[[t^s:s \in S]]$.
It is well known that the Cohen--Macaulay type of the numerical semigroup ring $k[[S]]$ is given by the cardinality of the set $\PF(S)$ of all pseudo-Frobenius numbers of $S$, where
$$
    \PF(S)=\{ x \in \Z \setminus  S \mid x+s \in S \text{~for all~} 0 \neq s \in S \}.
$$
The cardinality of $\PF(S)$ is called the type of $S$ and it is denoted by $\t(S)$. We say that $S$ is symmetric if $\t(S)=1$.
The smallest element of $S \setminus \{0\}$ is denoted by $\m(S)$, which is called the multiplicity of $S$. The largest element $\F(S)$ of $\Z \setminus S$ is called the Frobenius number of $S$. In \cite[Theorem 2.13]{arXiv:2306.17069}, the authors showed that the reduced type of a numerical semigroup ring $k[[S]]$ is equal to the cardinality of the set
$$
    \rPF(S) := \{ x \in \Z \setminus S : \F(S)-\m(S)+1 \le x \le \F(S) \}.
$$
The cardinality of $\rPF(S)$ is denoted by $s(S)$. Note that $\F(S) \in \rPF(S) \subset \PF(S)$, namely, $1\le s(S)\le \t(S)$. We say that $S$ has maximal reduced type if $s(S)=\t(S)$.

Almost symmetric numerical semigroups were introduced by Barucci and Froberg \cite{zbMATH01013931} and have been studied in several papers \cite{zbMATH06147883}, \cite{zbMATH06583408}, \cite{zbMATH06768981}, \cite{zbMATH07058163}. Answering Question above, this paper classifies almost symmetric numerical semigroups with maximal reduced type. For any positive integer $m$, the numerical semigroup $\Delta(m)=\{x \in \N \mid x \ge m\} \cup \{0\}$ is called a half-line. It is easy to show that if $S$ is symmetric or $S$ is a half-line, then $S$ is an almost symmetric numerical semigroup with maximal reduced type; so we determine almost symmetric numerical semigroups with maximal reduced type except for these cases as follows.

\begin{theorem}\label{MainTheorem}
    Let $S$ be a non-symmetric numerical semigroup
    with $\m(S)<\F(S)$. Then the following conditions are equivalent:
    \begin{itemize}
        \item[(1)] $S$ is an almost symmetric numerical semigroup with maximal reduced type;
        \item[(2)] There exist integers $t,m$ such that $2\le t<m$ and $S=\Delta(m) \setminus \{2m-t\}$.
    \end{itemize}
\end{theorem}

We note that $S$ is a half-line if and only if $\m(S)\ge \F(S)$.

\begin{corollary}\label{cor1}
    Let $S$ be a numerical semigroup. Then $S$ is an almost symmetric numerical semigroup with maximal reduced type if and only if one of the following conditions holds:
    \begin{itemize}
        \item[(1)] $S$ is a half-line;
        \item[(2)] $S$ is symmetric;
        \item[(3)] There exist integers $t,m$ such that $2\le t<m$ and $S=\Delta(m) \setminus \{2m-t\}$.
    \end{itemize}
\end{corollary}

This paper is organized as follows. In Section \ref{Preliminaries}, we recall some definitions and results about numerical semigroups. In Section \ref{Almost}, we prove Theorem \ref{MainTheorem} and show that if $t$ and $e$ are integers such that $2\le t\le e-1$, then there exists an almost symmetric numerical semigroup with maximal reduced type such that $\t(S)=t$ and $\e(S)=e$. In Section \ref{MED}, we classify MED-semigroups with maximal reduced type.

\section{Preliminaries}\label{Preliminaries}
First, we recall that some basic facts on numerical semigroups.
Let $\Z$ be the set of integers and let $\N=\{z\in \Z \mid z\ge 0\}$. A submonoid of $(\N,+)$ is a subset of $\N$ which is closed under the addition and contains the element $0$. A numerical semigroup is a submonoid $S$ of $(\N,+)$ such that $\N\setminus S$ is finite. Let $S$ be a numerical semigroup. The set $\N \setminus S$ is called the set of gaps of $S$.
Its cardinality is said to be the genus of $S$ and it is denoted by $\g(S)$. Let $A$ be a nonempty subset of $\N$. We denote by $\langle A \rangle$ the submonoid of $(\N,+)$ generated by $A$, that is,
$$
    \langle A \rangle =\{  \lambda_1a_1+\dots+\lambda_na_n \mid n\in
    \N \setminus \{0\}, \, a_1,\dots, a_n \in A \mbox{ and }
    \lambda_1,\dots,\lambda_n \in \N \}.
$$
It is well known that every numerical semigroup has a unique minimal system of generators and it is denoted by $\msg(S)$. Note that $\m(S)=\min(\msg(S))$. The cardinality $\e(S)$ of the minimal system of generators is called the embedding dimension of $S$. It is well known that the inequality $\e(S)\le\m(S)$ holds. We say that $S$ has maximal embedding dimension (MED-semigroup) if $\e(S)=\m(S)$.
In general, the following inequality holds (see \cite[Proposition 2.2]{zbMATH06147883}):
$$
    \g(S) \geq \frac{\F(S)+\t(S)}{2}.
$$

\begin{definition}
    A numerical semigroup $S$ is almost symmetric if $\g(S)=\frac{\F(S)+\t(S)}{2}$.
\end{definition}

Let $n \in S \setminus\{0\}$. We define the Ap{\'e}ry set of $n$ in $S$ as follows:
$$
    \Ap(S,n)=\{s \in S \mid s-n \notin S\}.
$$
Then we have the following result (see \cite[Lemma 2.4]{zbMATH05623301}):
$$
    \Ap(S,n)=\{0=w(0),w(1),\dots,w(n-1)\},
$$
where $w(i)$ is the least element of $S$ congruent with $i$ modulo $n$, for all $i \in \{0,\dots,n-1\}$. On $\Ap(S,n)$, we define the partial order $x\le_S y$ if $x+s=y$ for some $s \in S$. We denote the maximal elements of $\Ap(S,n)$ with respect to $\le_S$ by $\Maximals \Ap(S,n)$. We recall that
$$
    \PF(S)=\{ w-n \mid \Maximals\Ap(S,n)\}
$$
(see \cite[Proposition 2.20]{zbMATH05623301}). In particular, $\F(S)=\max\Ap(S,n)-n$. Moreover, if $S$ is not equal to $\N$, it is easy to see that
$$
    1\le \t(S) \le \m(S)-1.
$$
It is well known that when $S$ is not equal to $\N$, $S$ is a MED-semigroup if and only if $\t(S)=\m(S)-1$ (see \cite[Corollary 3.2]{zbMATH05623301}).

\section{Almost symmetric numerical semigroups with maximal reduced type}\label{Almost}
Our first goal in this section is to prove Theorem \ref{MainTheorem}. Let $t$ and $m$ be integers such that $2 \leq t < m$. It is clear that
$$
    \Delta(m) \setminus\{2m-t\}
$$
is a numerical semigroup. We show that this numerical semigroup is an almost symmetric numerical semigroup with maximal reduced type.\footnote{Similar results to Proposition \ref{prop1} were independently obtained by Kriti Goel, N\.{i}l \c{S}ah\.{i}n, Srishti Singh and Hema Srinivasan in \cite{goel2025numericalsemigroupssallytype}.}

\begin{proposition}\label{prop1}
    Let $t$ and $m$ be integers such that $2 \leq t < m$. We consider a numerical semigroup
    $$
        S= \Delta(m) \setminus\{2m-t\}.
    $$
    Then the following hold:
    \begin{itemize}
        \item[(1)] $\F(S)=2m-t, \m(S)=\g(S)=m$;
        \item[(2)] $\rPF(S) =\PF(S)=\{m-t+1,\dots,m-1,2m-t\}$;
        \item[(3)] $s(S)=\t(S)=t$;
        \item[(4)] $S$ is almost symmetric.
    \end{itemize}
    In particular, $S$ is an almost symmetric numerical semigroup with maximal reduced type.
\end{proposition}

\begin{proof}
    (1) is trivial and (3) follows from (2).

    (2) Since $\F(S)-\m(S)+1=m-t+1$, we have that
    $$
        \{m-t+1,\dots,m-1,2m-t\} = \rPF(S) \subset \PF(S).
    $$
    Let $x \in \PF(S)$. We may assume that $x\ge 0$. If $x \in \{1,\dots,m-t\}$, then
    $$
        \F(S)-x=2m-(t+x) \in \{m,\dots,2m-t-1\} \subset S \setminus \{0\}.
    $$
    According to the maximality of $x \in \PF(S)$, this is a contradiction. It follows that $x \ge m-t+1$, namely,
    $$
        \rPF(S) =\PF(S)=\{m-t+1,\dots,m-1,2m-t\}.
    $$

    (4) From (1), (3), we get
    $$
        \g(S)=\frac{\F(S)+\t(S)}{2}.
    $$
    This implies that $S$ is almost symmetric.
\end{proof}

\begin{example}
    Taking $m=7,t=4$ in Proposition \ref{prop1}, we obtain that $\rPF(S)=\PF(S)=\{4,5,6,10\}$.
\end{example}

\begin{lemma} \label{lemma1}
    Let $S$ be a numerical semigroup with $\m(S)< \F(S) < 2\m(S)$. Then
    $$
        \{  g \in \N\setminus S \mid \m(S)<g\le \F(S)\} \subset \rPF(S).
    $$
\end{lemma}

\begin{proof}
    Let $g \in \N \setminus S$ be such that $\m(S)<g \le \F(S)$.
    Since $\F(S) < 2\m(S)$, We obtain that $\F(S)-\m(S)+1 \le g$. Therefore, $g \in \rPF(S)$.
\end{proof}

\begin{proof}[Proof of Theorem \ref{MainTheorem}]
    $(2)\Rightarrow(1)$. It follows from Proposition \ref{prop1}.

    $(1)\Rightarrow(2)$. Since $S$ is not symmetric, there exists a pseudo-Frobenius number $f \in \PF(S)$ such that $f \neq \F(S)$. By \cite[Theorem 2.4]{zbMATH06147883}, we get $\F(S)-f \in \PF(S)$. So we may assume that $f \le \frac{\F(S)}{2}$. Since $S$ has maximal reduced type, we obtain that $\F(S) -\m(S)+1 \le  \frac{\F(S)}{2}$. Then
    $$
        \frac{\F(S)}{2} \le \m(S)-1.
    $$
    According to this, it follows that if $x$ is an integer such that $x \geq 2\m(S)$, then $x \in S$. We prove that if $x \in \Z \setminus \{\F(S)\}$ and $\m(S)<x<2\m(S)$, then $x \in S$. Assume by contradiction that there exists $x \in \Z \setminus \{\F(S)\}$ such that $\m(S)<x<2\m(S)$ and $x \notin S$. By Lemma \ref{lemma1}, we have that $x \in \rPF(S) \setminus \{\F(S)\}$. Since $S$ is an almost symmetric numerical semigroup with maximal reduced type, using \cite[Theorem 2.4]{zbMATH06147883}, we obtain that $\F(S)-x \in \rPF(S)$. Hence, $\F(S)-\m(S)+1 \le \F(S)-x$. This implies that $x \le \m(S)-1$, which is a contradiction. Then
    $$
        S=\Delta(\m(S)) \setminus \{\F(S)\}.
    $$
    Note that $\g(S)=\m(S)$. Since $S$ is almost symmetric, it follow that $\m(S)= \frac{\F(S)+\t(S)}{2}$. Therefore,
    $$
        S= \Delta(\m(S)) \setminus \{2\m(S)-\t(S)\}.
    $$
\end{proof}

\begin{remark}
    The inequality  $\frac{\F(S)}{2} \le \m(S)-1$ in the proof of Theorem \ref{MainTheorem} is proved in \cite[Proposition 3.16(1)]{arXiv:2306.17069}.
\end{remark}

Our next aim is to show that for any integers $t,e$ with $2\le t\le e-1$, there exists an almost symmetric numerical semigroup with maximal reduced type such that $\t(S)=t$ and $\e(S)=e$. The following fact is well known and appears in \cite{zbMATH05623301}.

\begin{fact}\label{fact1}
    Let $S$ be a numerical semigroup such that $S \neq\N$. Then the following hold:
    \begin{itemize}
        \item[(1)] $\msg(S)=(S \setminus\{0\}) \setminus \{s+t \mid s,t \in S \setminus\{0\}\}$;
        \item[(2)] If $s,t \in S$ and $s+t \in \Ap(S,\m(S))$, then $s,t \in \Ap(S,\m(S))$;
        \item[(3)] $\msg(S)  \subset (\Ap(S,\m(S)) \setminus\{0\})\cup \{\m(S)\}$.
    \end{itemize}
\end{fact}

\begin{lemma}\label{lemma2}
    Let $S$ be a numerical semigroup such that $S \neq \N$, let $n \in S \setminus \{0\}$ and let $\msg(S)=\{\m(S)=n_1 < \cdots<n_e\}$. If $n \in \Ap(S,\m(S)) \setminus \{0\}$ and $n \leq 2n_2 -1$, then $n \in \msg(S)$.
\end{lemma}

\begin{proof}
    Let $n \in \Ap(S,\m(S))\setminus\{0\}$ be such that $n \leq 2n_2 -1$. Assume that $n \notin \msg(S)$. By Fact \ref{fact1}(1), $n=s+t$ for some $s,t \in S \setminus\{0\}$. Since $n \in \Ap(S,\m(S))$, using Fact \ref{fact1}(2), it follows that $s,t \in \Ap(S,\m(S)) \setminus\{0\}$. Since $\min \Ap(S,\m(S)) \setminus\{0\}=n_2$, we have that $n_2 \le s, n_2 \le t$. This implies that
    $$
        n=s+t \ge 2n_2,
    $$
    which is a contradiction. Therefore, $n \in \msg(S)$.
\end{proof}

\begin{proposition}\label{prop2}
    Let $t$ and $m$ be  integers such that $2 \leq t < m$ and let
    $$
        S= \Delta(m) \setminus\{2m-t\}.
    $$
    Then,
    $$
        \Ap(S,m)=
        \begin{cases}
            \{0,m+2,\dots,2m-1,2m+1\}                            & \quad(t=m-1)    ; \\
            (\{0,m+1,\dots,2m-1\}\setminus\{2m-t\})\cup \{3m-t\} & \quad(t\le m-2).  \\
        \end{cases}
    $$
    Moreover,
    $$
        \msg(S)=  \begin{cases}
            \{m,m+2,\dots,2m-1,2m+1\}             & \quad(t=m-1);     \\
            \{m,m+1,\dots,2m-1\}\setminus\{2m-t\} & \quad(t \le m-2).
        \end{cases}
    $$
    In particular,
    $$
        \e(S)=
        \begin{cases}
            m   & \quad(t=m-1);     \\
            m-1 & \quad(t \le m-2). \\
        \end{cases}
    $$
\end{proposition}

\begin{proof}
    Let $\msg(S)=\{m=n_1<\cdots<n_e\}$.

    ($t=m-1$). By the assumption, we have that $S=\Delta(m)\setminus\{m+1\}$. Then, $n_2=m+2$ and
    $$
        \Ap(S,m)=\{0,m+2,\dots,2m-1,2m+1\}.
    $$
    It follows from Fact \ref{fact1}(3) and Lemma \ref{lemma2} that
    $$
        \msg(S)=\{m,m+2,\dots,2m-1,2m+1\},
    $$
    namely, $\e(S)=m$.

    ($t \le m-2$). By $2\le t\le m-2$, we have that $m+2 \le 2m-t \le 2m-2$.
    It follows that
    $$
        \Ap(S,m)= (\{0,m+1,\dots,2m-1\}\setminus\{2m-t\})\cup \{3m-t\}.
    $$
    By Fact \ref{fact1}(1), we have that $3m-t=(m+1)+(2m-t-1) \notin \msg(S)$. Similarly to the case $t=m-1$, we obtain that
    $$
        \msg(S)=\{m,m+1,\dots,2m-1\}\setminus\{2m-t\},
    $$
    namely, $\e(S)=m-1$.
\end{proof}

We obtain the following corollary from Corollary \ref{cor1} and Propositions \ref{prop1}, \ref{prop2}.

\begin{corollary}
    Let $S$ be a numerical semigroup such that $S\neq \N$. If $S$ is an almost symmetric numerical semigroup with maximal reduced type, then
    $$
        \t(S) \le \e(S)-1.
    $$
\end{corollary}

\begin{corollary}
    Let $t$ and $e$ be integers such that $2\leq t \leq e-1$. Then, there exists an almost symmetric numerical semigroup with maximal reduced type such that $\t(S)=t$ and $\e(S)=e$.
\end{corollary}

\begin{proof}
    Let $m=e+1$. Note that $2\le t\le m-2$. Set $S=\Delta(m) \setminus\{2m-t\}$. By Proposition \ref{prop1} and Proposition \ref{prop2}, it follows that $S$ is an almost symmetric numerical semigroup with maximal reduced type such that $\t(S)=t$ and $\e(S)=e$.
\end{proof}

This corollary recovers \cite[Theorem 2.10]{zbMATH05316270}.

\begin{remark}\label{remark}
    Let $t$ and $m$ be integers such that $1 \leq t < m$. It is natural to consider the numerical semigroup $S=\Delta(m)\setminus\{2m-t\}$, where $t=1$. If $t=m-1$, then $S=\langle 2,5 \rangle$. Hence, it is enough to focus on the case $1=t\le m-2$.
    Using the same argument in the proof of Proposition \ref{prop1}, it follows that $S$ is a symmetric numerical semigroup with Frobenius number $2m-1$. This result is already pointed out in \cite[Lemma 2.8]{zbMATH05316270}. Moreover, applying the same argument in the proof of Proposition \ref{prop2}, we have the following statements:
    \begin{itemize}
        \item [(1)] $\Ap(S,m)=\{0,m+1,\dots,2m-2,3m-1\}$;
        \item [(2)] $\msg(S)=\{m,m+1,\dots,2m-2\}$;
        \item [(3)] $\e(S)=m-1$.
    \end{itemize}
\end{remark}

By Proposition \ref{prop1} and Remark \ref{remark}, we have the next result.

\begin{corollary}
    For any integers $m,F$ such that $2\le m<F<2m$, $\Delta(m) \setminus\{F\}$ is an almost symmetric numerical semigroup with maximal reduced type such that $\t(S)=2m-F$.
\end{corollary}

\section{MED-semigroups with maximal reduced type}\label{MED}

In this section, we classify MED-semigroups with maximal reduced type.

\begin{lemma}\label{lemma3}
    Let $S$ be a numerical semigroup such that $S \neq \N$.
    Then,
    $$
        1\le s(S)\le \t(S)\le \m(S)-1.
    $$
\end{lemma}

\begin{proof}
    We recall the inequality
    $$
        1\le  \t(S)\le\m(S)-1.
    $$
    Since $1\le s(S)\le \t(S)$, it follows that
    $$
        1\le s(S)\le \t(S)\le \m(S)-1.
    $$
\end{proof}

\begin{lemma}\label{lemma4}
    Let $S$ be a numerical semigroup. Then,
    $$
        \{w \in \Ap(S,\m(S))\mid w\ge \F(S)+1 \} =\{h+\m(S) \mid h \in \rPF(S)\}.
    $$
    In particular,
    $$
        \{w \in \Ap(S,\m(S))\mid w\ge \F(S)+1 \} \subset \Maximals\Ap(S,\m(S)).
    $$
\end{lemma}

\begin{proof}
    Let $w \in \Ap(S,\m(S))$ with $w\ge \F(S)+1$. Then $w-\m(S) \notin S$ and $w-\m(S) \ge \F(S)-\m(S)+1$, which implies that $w-\m(S) \in \rPF(S)$. Conversely, let $h \in \rPF(S)$. Since $h +\m(S) \ge \F(S)+1$, $h+\m(S) \in S$. Therefore, we obtain that $h+\m(S) \in \Ap(S,\m(S))$.
\end{proof}

\begin{proposition}\label{prop3}
    Let $S$ be a numerical semigroup such that $S \neq \N$, and let $\msg(S)=\{\m(S)=n_1<\cdots<n_e\}$. Then the following conditions are equivalent:
    \begin{itemize}
        \item [(1)] $s(S)=\m(S)-1$;
        \item [(2)] $S$ is a MED-semigroup with maximal reduced type;
        \item [(3)] $\F(S)+1\le n_2$.
    \end{itemize}
    In particular, if $S$ is a half-line, then $S$ is a MED-semigroup with maximal reduced type.
\end{proposition}

\begin{proof}
    $(1) \Rightarrow (2)$. It follows from Lemma \ref{lemma3}.

    $(2) \Rightarrow (3)$. See \cite[Corollary 3.10]{arXiv:2306.17069}.

    $(3)\Rightarrow(1)$. We note that $\min(\Ap(S,\m(S))\setminus\{0\})=n_2$. By the assumption, we obtain that
    $$
        \Ap(S,\m(S))\setminus\{0\} \subset \{w \in \Ap(S,\m(S))\mid w \ge \F(S)+1\}.
    $$
    By Lemma \ref{lemma4}, we have $\m(S)-1\le s(S)$, namely, $s(S)=\m(S)-1 $.
\end{proof}

Let $m$ and $F$ be integers such that $2\le m<F$ and $m$ does not divide $F$. We define
$$
    \Delta(F,m):=\langle m \rangle \cup \{x \in \N \mid x\ge F+1\}.
$$
It is known that $\Delta(F,m)$ is a MED-semigroup with Frobenius number $F$ and multiplicity $m$ (see \cite[Lemma 2.1]{zbMATH08100532}).

\begin{theorem}
    Let $S$ be a numerical semigroup with $\m(S)<\F(S)$  and let $\msg(S)=\{\m(S)=n_1<\cdots<n_e\}$. Then the following conditions are equivalent:
    \begin{itemize}
        \item[(1)] $S$ is a MED-semigroup with maximal reduced type;
        \item[(2)] There exist integers $F,m$ such that $2\le m<F$, $m$ does not divide $F$ and $S=\Delta(F,m)$.
    \end{itemize}
\end{theorem}

\begin{proof}
    $(2)\Rightarrow(1)$. Note that $\F(S)+1\le n_2$.
    By Proposition \ref{prop3}, it follows that $\Delta(F,m)$ is a MED-semigroup with maximal reduced type.

    $(1)\Rightarrow(2)$. Since $\m(S)<\F(S)$, we obtain that $\m(S)\ge2$. Note that $\m(S)$ does not divide $\F(S)$. Clearly, $\Delta(\F(S),\m(S))\subset S$. By Proposition \ref{prop3}, $n_2 \ge \F(S)+1$. This inequality implies that
    $\{\m(S)=n_1<\cdots<n_e\} \subset \Delta(\F(S),\m(S))$. Then, $S \subset \Delta(\F(S),\m(S))$, namely,
    $S=\Delta(\F(S),\m(S)).$
\end{proof}

\begin{corollary}
    Let $S$ be a numerical semigroup. Then $S$ is a MED-semigroup with maximal reduced type if and only if one of the following conditions holds:
    \begin{itemize}
        \item[(1)] $S$ is a half-line;
        \item[(2)] There exist integers $F,m$ such that $2\le m<F$, $m$ does not divide $F$ and $S=\Delta(F,m)$.
    \end{itemize}
\end{corollary}

\section*{Acknowledgments}
I would like to thank Professor Satoru Fukasawa for valuable comments and suggestions.


\end{document}